\documentclass[10pt,reqno]{amsart}
\usepackage{enumitem,amsmath, amsthm, amssymb, amsfonts, mathtools,fullpage} 
\usepackage{graphicx, tikz}
\usepackage{hyperref}
\usepackage{multirow}
\usepackage{epigraph}
\usepackage{mathtools}
\usepackage{fullpage}
\usepackage{subfig}
\usepackage{color}
\usepackage{tikz}
\usepackage{adjustbox}
\usepackage{enumitem}
\usepackage{float}
\usetikzlibrary{arrows}
\theoremstyle{definition}

\newtheorem{theorem}{Theorem}[section]
\newtheorem{lemma}[theorem]{Lemma}

\newtheorem{corollary}[theorem]{Corollary}

\newtheorem{problem}{Problem}

\newtheorem{innercustomthm}{Theorem}
\newenvironment{mythm}[1]
  {\renewcommand\theinnercustomthm{#1}\innercustomthm}
  {\endinnercustomthm}

\newcommand{\VPF}{\mathrm{VPF}}
\newcommand{\PF}{\mathrm{PF}}

\title{Vacillating parking functions}

\author{Bruce Fang}
\author{Pamela E. Harris}
\address[P.~E. Harris]{Department of Mathematical Sciences, University of Wisconsin-Milwaukee, Milwaukee, WI 53211}
\email{\textcolor{blue}{\href{mailto:peharris@uwm.edu}{peharris@uwm.edu}}}
\author{Brian M. Kamau}
\author{David Wang}
\address[B.~Fang, B.~M.~Kamau, D. Wang]{Department of Mathematics and Statistics, Williams College, Williamstown, MA 01267}
\email{\textcolor{blue}{\href{mailto:bf8@williams.edu}{bf8@williams.edu}}, 
\textcolor{blue}{\href{mailto:bmk5@williams.edu}{bmk5@williams.edu}}, 
\textcolor{blue}{\href{mailto:dw8@williams.edu}{dw8@williams.edu}}}

\begin{document}

\begin{abstract}
    For any integers $1\leq k\leq n$, we introduce a new family of parking functions called $k-$\textit{vacillating parking functions} of length $n$. 
    The parking rule for $k$-vacillating parking functions allows a car with preference $p$ to park in the first available spot in encounters among the parking spots numbered $p$, $p-k$, and $p+k$ (in that order and if those spots exists).
    In this way,  $k$-vacillating parking functions are a modification of Naples parking functions, which allow for backwards movement of a car, and of $\ell$-interval parking functions, which allow a car to park in its preference or up to $\ell$ spots in front of its preference.  
    Among our results, we establish a combinatorial interpretation for the numerator of the $n$th convergent of the continued fraction of $\sqrt{2}$, as the number of non-decreasing $1$-vacillating parking functions of length~$n$.
    Our main result gives a product formula for the enumeration of $k$-vacillating parking functions of length $n$ based on the number of $1$-vacillating parking functions of smaller length. 
    We conclude with some directions for further research.
\end{abstract}

\maketitle

\section{Introduction}

Fix $n\in\mathbb{N}\coloneqq\{1,2,3,\ldots\}$ and let $[n]\coloneqq\{1,2,3,\ldots,n\}$. 
Consider $n$ cars in order from 1 to $n$ entering a one-way street consisting of $n$ consecutively numbered parking spots. 
For each $i\in[n]$, car $i$ has a preferred parking spot $a_i\in[n]$
and we encode their preferences in a parking list $(a_1,a_2,\ldots,a_n)$. 
Given a preference list $\alpha=(a_1,a_2,\ldots,a_n)\in[n]^n$, for $i=1,2,\ldots,n$, car $i$ drives down the street and parks in the first available spot at or past its preference $a_i$, if such a parking spot exists. 
If all cars are able to park, then we say that $\alpha=(a_1,a_2,\ldots,a_n)$ is a \textit{parking function} (of length $n$).
Throughout we let $\PF_n$ denote the set of parking functions of length $n$.
For example, $(3,2,4,1,1)\in\PF_5$ is a parking function, in which car 1 parks in spot 3, car 2 parks in spot 2, car 3 parks in spot 4, car 4 parks in spot 1, and car 5 parks in spot 5. However, $(3,2,4,4,4)$ is not a parking function as car 5 is unable to park on the street.

Parking functions were introduced by Konheim and Weiss in their study of hashing functions~\cite{Konheim1966AnOD}. 
They established that $|\PF_n| = (n + 1)^{n-1}$. Since their foundational result, parking functions have been shown to have connections to numerous areas of mathematical interest. For example, parking functions arise in the study of volumes of flow polytopes, hyperplane arrangements, ideal states in the game of the Tower of Hanoi, and even in the classical divide-and-conquer sorting algorithm $\mathrm{Quicksort}$ \cite{Aguillon2022OnPF, Benedetti2018ACM, Harris2023LuckyCA, Stanley1996HyperplaneAI}.
There are also numerous generalizations of parking functions with scenarios in which
\begin{itemize}[leftmargin=.25in]
    \item cars have assorted lengths, which are called parking sequences and parking assortments \cite{Chen_2023,MR3593646,franks2023counting};
    \item cars arriving later in the queue can bump earlier cars from their parking spot, called MVP parking functions \cite{MVP};
    \item cars can back up when finding their preferred spot occupied, called Naples parking functions \cite{Christensen2019AGO,colmenarejo2021counting};
    \item cars are restricted to park only in a subset of the spots on the street, called subset parking functions, and another called $\ell$-interval parking functions, which allows cars to park at most $\ell$ spots ahead of their preference~\cite{aguilarfraga2023interval,Hadaway2021GeneralizedPF, Spiro2019SubsetPF}.
\end{itemize}  
We point the reader to Yan ~\cite{yan2015parking}  for a comprehensive survey of results related to parking functions and for those interested in open problems in related to parking functions we recommend \cite{Choose}.

In this work, we describe a new parking rule based on a parameter $k\in[n]$, which we call the $k$\textit{-vacillating  parking rule}. This parking rule is described as follows: 
Given the preference list $\alpha=(a_1,a_2,\ldots,a_n)\in[n]^n$, cars enter in order $i=1,2,\ldots,n$. 
Upon the arrival of car $i \in [n]$ it attempts to park in its preferred spot $a_i$. If spot $a_i$ is unoccupied, it parks there. 
Otherwise, car $i$ backs up to check spot $a_i-k$ and parks there if it is available. If spot $a_i-k$ does not exist or is occupied, then car $i$ drives forward to spot $a_i+k$ (if it exists) and parks there if it is available.
If the preference list $\alpha=(a_1, \ldots, a_n)\in [n]^n$ allows all of the cars to park under the $k$-vacillating parking rule, then we say $\alpha$ is a \textit{$k$-vacillating parking function} of length $n$. 
    We let $\VPF_{n}(k)$ denote the set of $k$-vacillating parking functions of length $n$. When $k=1$ we call the $1$-vacillating parking functions simply vacillating parking functions and we further 
    simplify notation     by letting $\VPF_n\coloneqq\VPF_{n}(1)$.

    As illustrated in Figure~\ref{fig:ps example}, under the $2$-vacillating parking rule: car 1 prefers spot $4$ and parks in spot $4$,  car 2 prefers spot $1$ and parks in spot $1$. Note car 3 also prefers spot $1$ and since spot $1$ is already occupied and spot $1-2=-1$ does not exist, then car 3 parks in spot $1+2=3$. Similarly, car 4 prefers spot $4$ and parks in spot $4-2=2$. Since all the car can park, then $(4, 1, 1, 4)\in \VPF_{4}(2)$. 
However, one can confirm that the preference list $(4, 1, 1, 1)\in [4]^4$, as illustrated in Figure~\ref{fig:ps example}, has the first three cars parking in spots $4$, $1$, and $3$, respectively. Car 4 prefers spot $1$, and since spot $1$ and spot $3$ are both occupied, then car 4 leaves without parking. Thus, $(4, 1, 1, 1)\notin \VPF_{4}( 2)$.

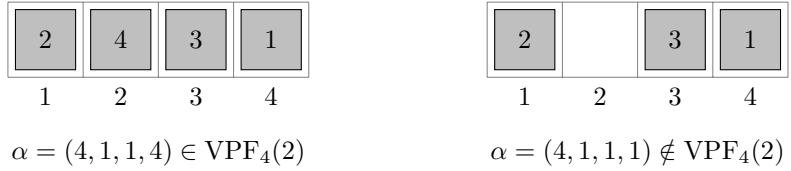
\begin{figure}[!h]
    \centering
    \begin{tikzpicture}
\node at (2,-1) {$\alpha=(4,1,1,4)\in\VPF_4 ( 2)$};
    \draw[step=1cm,gray,very thin] (0,0) grid (4,1);
    \draw[fill=gray!50] (.1,0.1) rectangle (.9,.9);
    \node at (.5,.5) {$2$};
    \draw[fill=gray!50] (1.1,0.1) rectangle (1.9,.9);
    \node at (1.5,.5) {$4$};
    \draw[fill=gray!50] (2.1,0.1) rectangle (2.9,.9);
    \node at (2.5,.5) {$3$};
    \draw[fill=gray!50] (3.1,0.1) rectangle (3.9,.9);
    \node at (3.5,.5) {$1$};
    \node at (.5,-.25) {$1$};
    \node at (1.5,-.25) {$2$};
    \node at (2.5,-.25) {$3$};
    \node at (3.5,-.25) {$4$};
    \end{tikzpicture}
    \qquad\qquad\qquad
    \begin{tikzpicture}
\node at (2,-1) {$\alpha=(4,1,1,1)\notin\VPF_4( 2)$};
    \draw[step=1cm,gray,very thin] (0,0) grid (4,1);
    \draw[fill=gray!50] (.1,0.1) rectangle (.9,.9);
    \node at (.5,.5) {$2$};
    \draw[fill=gray!50] (2.1,0.1) rectangle (2.9,.9);
    \node at (2.5,.5) {$3$};
    \draw[fill=gray!50] (3.1,0.1) rectangle (3.9,.9);
    \node at (3.5,.5) {$1$};
    \node at (.5,-.25) {$1$};
    \node at (1.5,-.25) {$2$};
    \node at (2.5,-.25) {$3$};
    \node at (3.5,-.25) {$4$};
    \end{tikzpicture}
    \caption{Parking position of cars under the $2$-vacillating parking rule.
}
    \label{fig:ps example}
\end{figure}

Our main result gives a product formula for the number of $k$-vacillating parking functions of length~$n$, which depends solely on the enumeration of the vacillating parking functions.

\begin{mythm}{4.1}
Fix integers $1\leq k\leq n$. If $n=ka+b$, with $0\leq a$ and $0\leq b\leq k-1$, then \[|\VPF_n(k)| \ = \left(\
\frac{n!}{\prod_{t=0}^{k-1}\lfloor\frac{n+t}{k}\rfloor!}\right)\cdot|\VPF_{a+1}|^b\cdot|\VPF_{a}|^{k-b},\]
where $|\VPF_{0}|\coloneqq1$.
\end{mythm}
We now state a recursive formula for the number of vacillating parking functions, which we prove in Section~\ref{mp 1}.

\begin{mythm}{2.1}
If $n\geq 1$, then the number of vacillating parking functions of length $n$, denoted by $|\VPF_n|$,  satisfies the recurrence relations:
\begin{align*}
|\VPF_n|&\ = \ \sum_{i=1}^n|\VPF_{n,i}|,
\\
|\VPF_{n,i}|&\ = \ |\VPF_{n,(i)}|+|\VPF_{n,[i]}|,
\\
    |\VPF_{n,(i)}|&\ = \ |\VPF_{n-1}|+(n-i)|\VPF_{n-2}|+\sum_{l=i}^{n-2}(l+1-i)|\VPF_{n-2, (l)}|,
\\
|\VPF_{n,[i]}|\ &= \ \frac{(i-1)(i-2)}{2}|\VPF_{n-3}|+\sum_{l=1}^{i-1}|\VPF_{n-1,[l]}|+(i-1)\sum_{l=1}^{i-2}|\VPF_{n-2,l}|\\
&\qquad\qquad+\sum_{l=1}^{i-3}\frac{l(1+l)}{2}|\VPF_{n-3, (l)}|+\frac{(i-2)(i-1)}{2}\sum_{l=i-2}^{n-3}|\VPF_{n-3, (l)}|
\end{align*}
where 
\begin{itemize}[leftmargin=.2in]
    \item $\VPF_{n,i}$ denotes the set of vacillating parking functions of length $n$ where spot $n$ is occupied by car~$i$;
    \item $\VPF_{n,(i)}$ denotes the set of vacillating parking functions of length $n$ where car $i$ prefers spot $n$ and parks in spot $n$;
    \item $\VPF_{n,[i]}$ denotes the set of vacillating parking functions of length $n$, where car $i$ prefers spot $n-1$ and parks in spot $n$.
\end{itemize}
The initial conditions are given by 
\begin{itemize}[leftmargin=.15in]
\item $|\VPF_{1, (1)}|=|\VPF_{2, (2)}|=1$,~$|\VPF_{2, (1)}|=2$,~$|\VPF_{3, (1)}|=7$,~$|\VPF_{3, (2)}|=5$,~$|\VPF_{3, (3)}|=4$;
\item $|\VPF_{1, [1]}|=|\VPF_{2, [1]}|=|\VPF_{3, [1]}|=|\VPF_{3, [2]}|=0$,~$|\VPF_{2, [2]}|=1$, $|\VPF_{3, [3]}|=4$.\end{itemize}
\end{mythm}
In Section~\ref{mmp1}, we consider the subsets of vacillating parking functions consisting of non-decreasing and non-increasing tuples, which we denote by 
$|\VPF_{n}^{\uparrow}|$
and $|\VPF_{n}^{\downarrow}|$, respectively. 
Our main results are: 
\begin{enumerate}
    \item Theorem~\ref{cardinality}: If $n\geq 1$, then
\[|\VPF_n^{\uparrow}| \ = \ 2\cdot|\VPF_{n-1}^{\uparrow}| + |\VPF_{n-2}^{\uparrow}|,\]
with $|\VPF_1^{\uparrow}|=1$ and $|\VPF_2^{\uparrow}|=3$.
\item Theorem~\ref{non-increasing}: If $n\geq 1$, then
\[|\VPF_n^{\downarrow}|\ = \ 2\cdot|\VPF_{n-1}^{\downarrow}| + |\VPF_{n-3}^{\downarrow}|,\]
with $|\VPF_1^{\downarrow}|=1$, $|\VPF_2^{\downarrow}|=3$, and $|\VPF_3^{\downarrow}|=6$.
\end{enumerate}
We remark that
the numerator of the $n$th convergent of the continued fraction of $\sqrt{2}$ satisfies the same recurrence as that presented in Theorem~\ref{cardinality}. Thus, our result 
provides a combinatorial interpretation for these numerators as the number of non-decreasing vacillating parking functions of length~$n$. We provide more details and some historical context on continued fractions in Section~\ref{continued fractions}.

In Section~\ref{Generalization} we prove our main result Theorem~\ref{genearlgeneral} and give the product formula for the number of $k$-vacillating parking functions of length $n$. We conclude the article in Section \ref{sec:future}, where we state some open problems for further study.

\section{Vacillating Parking Functions}\label{mp 1}
In this section, we prove the recursive formula for the number of vacillating parking functions by enumerating the following subsets:
\begin{itemize}[leftmargin=.2in]
    \item $\VPF_{n,i}$ is the set of vacillating parking functions of length $n$ where spot $n$ is occupied by car~$i$;
    \item $\VPF_{n,(i)}$ is the set of vacillating parking functions of length $n$ where car $i$ prefers spot $n$ and parks in spot $n$;
    \item $\VPF_{n,[i]}$ is the set of vacillating parking functions of length $n$, where car $i$ prefers spot $n-1$ and parks in spot $n$.
\end{itemize}

\begin{theorem}
\label{thm:multirecurrence}
Let $n\in\mathbb{N}$. Then the recursive formulas for $|\VPF_n|$ are given by
\begin{align}
|\VPF_n|&\ = \ \sum_{i=1}^n|\VPF_{n,i}|,\label{V1}
\\
|\VPF_{n,i}|&\ = \ |\VPF_{n,(i)}|+|\VPF_{n,[i]}|,\label{V2}
\\
    |\VPF_{n,(i)}|&\ = \ |\VPF_{n-1}|+(n-i)|\VPF_{n-2}|+\sum_{\ell=i}^{n-2}(\ell+1-i)|\VPF_{n-2, (\ell)}|,\label{V3}
\\
|\VPF_{n,[i]}|\ &= \ \frac{(i-1)(i-2)}{2}|\VPF_{n-3}|+\sum_{\ell=1}^{i-1}|\VPF_{n-1,[\ell]}|+(i-1)\sum_{\ell=1}^{i-2}|\VPF_{n-2,\ell}|\nonumber\\
&\qquad\qquad+\sum_{\ell=1}^{i-3}\frac{\ell(\ell+1)}{2}|\VPF_{n-3, (\ell)}|+\frac{(i-2)(i-1)}{2}\sum_{\ell=i-2}^{n-3}|\VPF_{n-3, (\ell)}|\label{V4},
\end{align}
with initial conditions 
\begin{itemize}[leftmargin=.2in]
\item $|\VPF_{1, (1)}|=|\VPF_{2, (2)}|=1$,$|\VPF_{2, (1)}|=2$,~$|\VPF_{3, (1)}|=7$,~$|\VPF_{3, (2)}|=5$,~$|\VPF_{3, (3)}|=4$;
\item $|\VPF_{1, [1]}|=|\VPF_{2, [1]}|=|\VPF_{3, [1]}|=|\VPF_{3, [2]}|=0$,~$|\VPF_{2, [2]}|=1$,~and~$|\VPF_{3, [3]}|=4$.
\end{itemize}
\end{theorem}

\begin{proof}
Enumerating the elements of $\VPF_{n}$ is equivalent to summing over the enumerations of vacillating parking functions of length $n$ where spot $n$ is occupied by car $i$, where $1\leq i\leq n$. Hence, we arrive at~\eqref{V1}:
\begin{equation*}
|\VPF_n| \ = \ \sum_{i=1}^n|\VPF_{n,i}|.
\end{equation*}
Moreover, since spot $n$ can only be occupied by a car that prefers spot $n$ or spot $n-1$, then we arrive at~\eqref{V2}:
\begin{equation*}
|\VPF_{n,i}| \ = \ |\VPF_{n,(i)}|+|\VPF_{n,[i]}|,
\end{equation*}
since  $\VPF_{n,(i)}$ is the set of vacillating parking functions of length $n$ where car $i$ prefers spot $n$ and parks in spot $n$; and $\VPF_{n,[i]}$ is the set of vacillating parking functions of length $n$, where car $i$ prefers spot $n-1$ and parks in spot $n$.

Next we show that $|\VPF_{n,(i)}|$ is enumerated by equation~\eqref{V3}. 
Since at most two cars can prefer spot $n$, we consider the following three cases:
\begin{enumerate}[leftmargin=.3in]
\item Two cars prefer spot $n$ and no cars prefer spot $n-1$.\\
Since  car $i$ prefers spot $n$ and parks in spot $n$, then it must be the first among the two cars that prefer spot $n$. Then, any single one of the $n-i$ cars after the  car $i$, could prefer spot $n$ but must park in spot $n-1$. Hence, there are $n-i$ possibilities for the second car that prefers spot $n$. Since no cars prefer spot $n-1$, 
and whichever second car has preference $n$, then the number of ways to park the rest of the $n-2$ cars is given by $|\VPF_{n-2}|$. Hence, in total there are $(n-i)|\VPF_{n-2}|$ such parking functions.
\item Two cars prefer spot $n$ and some other car prefers spot $n-1$.\\
By Case 1's argument, car $i$ is the first car that prefers spot $n$ and parks there. Let the second car that prefers spot $n$ be the  car $j$. Then $i+1\leq j\leq n$. Since some other car $\ell$ prefers spot $n-1$ and takes spot $n-2$, then the car $\ell$ has to be in queue after  car $j$, as otherwise  car $\ell$ would park in spot $n-1$ and car $j$ would have been unable to park. Hence, for each $1\leq i\leq n$ and  $i+1\leq j\leq n$, the number of ways to park the rest of the $n-2$ cars with car $\ell$ preferring spot $n-1$ and taking spot $n-2$ is exactly identical to the number of ways to park $n-2$ cars with  car $\ell$ preferring and taking spot $n-2$, which is equal to
\begin{equation*}
\sum_{\ell=j-1}^{n-2}|\VPF_{n-2, (\ell)}|.
\end{equation*}
Notice that the summation goes from $\ell=j-1$ to $n-2$ instead of from $\ell=j+1$ to $n$, as we must reindex to account for the fact that car $i$ and car $j$ precede  car $\ell$ in the queue. Now accounting for the fact that $i+1\leq j\leq n$, yields
\begin{equation*}
\sum_{j=i+1}^{n}\sum_{\ell=j-1}^{n-2}|\VPF_{n-2, (\ell)}|.   
\end{equation*}
Expanding the summations and grouping common terms, we arrive at the expression
\begin{equation*}
\sum_{\ell=i}^{n-2}(\ell+1-i)|\VPF_{n-2, (\ell)}|.
\end{equation*}
\item Only one car prefers spot $n$.\\
In this case there are $|\VPF_{n-1}|$ such parking functions.
\end{enumerate}
Adding the terms from each of the three cases above yields equation \eqref{V3}. 

Next, we show that $|\VPF_{n,[i]}|$ can be enumerated using equation~\eqref{V4}. Recall that $\VPF_{n,[i]}$ denotes the set of vacillating parking functions of length $n$, where car $i$ prefers spot $n-1$ and parks in spot $n$. We again consider three cases based on the number of cars that prefer spot $n-1$:
\begin{enumerate}[leftmargin=.3in]
\item Three cars prefer spot $n-1$.\\
Let the first car preferring spot $n-1$ be  car $h$, and the second such car be the car $j$. Since the  car $i$ takes spot $n$, then it must be the third car that prefers spot $n-1$. Hence,  car $h$ parks in spot $n-1$, car $j$ parks in spot $n-2$, and the car $i$ parks in spot $n$. 
So any car $\ell$ (if they exist) that prefers spot $n-2$ and parks in spot $n-3$ must satisfy $\ell>j$. Otherwise, the car $\ell$ would park in spot $n-2$ and car $j$ would park in spot $n$, which disallows the car $i$ from parking in spot $n$, as required.

When there exists such a car $\ell$, we can enumerate based on the relative positions of car $i$ and car $\ell$. Since the car $h$ precedes car $j$, which is succeeded by the car $\ell$, then if  car $\ell$ precedes car $i$, we have that $2\leq j\leq i-2$, $1\leq h\leq j-1$, and $j+1\leq \ell\leq i-1$. 
Since car $\ell$ prefers spot $n-2$ and parks in spot $n-3$, which is equivalent to  car $\ell$ preferring spot $n-3$ and taking spot $n-3$, then the number of parking functions in the case where car $k$ precedes car $i$  is given by
\begin{equation*}
\sum_{j=2}^{i-2}\sum_{h=1}^{j-1}\sum_{\ell=j-1}^{i-3}|\VPF_{n-3,(\ell)}|\ = \ \sum_{j=2}^{i-2}\sum_{\ell=j-1}^{i-3}(j-1)|\VPF_{n-3,(\ell)}|.
\end{equation*}
Notice that the summation goes from $\ell=j-1$ to $i-3$ instead of from $\ell=j+1$ to $i-1$ because  car $h$ and  car $j$ precede the  car $\ell$. Moreover, simplifying yields
\begin{equation*}
\sum_{j=2}^{i-2}\sum_{\ell=j-1}^{i-3}(j-1)|\VPF_{n-3,(\ell)}| \ = \ \sum_{j=2}^{i-1}\sum_{\ell=j-1}^{i-3}(j-1)|\VPF_{n-3,(\ell)}|
\end{equation*}
since if $j=i-1$, then $l=i-2>i-3$ and this summand does not exist. 

If the car $\ell$ succeeds car $i$, then $2\leq j\leq i-1$, $1\leq h \leq j-1$, and $i+1\leq \ell\leq n$. In this case, the number of parking functions when car $\ell$ succeeds car $i$ is given by
\begin{equation*}
\sum_{j=2}^{i-1}\sum_{h=1}^{j-1}\sum_{\ell=i-2}^{n-3}|\VPF_{n-3,(\ell)}|\ = \ \sum_{j=2}^{i-1}\sum_{\ell=i-2}^{n-3}(j-1)|\VPF_{n-3,(\ell)}|.
\end{equation*}
Thus,
when some car prefers spot $n-2$, the number of such parking functions is
\begin{equation*}
\sum_{j=2}^{i-1}\sum_{\ell=j-1}^{i-3}(j-1)|\VPF_{n-3,(\ell)}|+\sum_{j=2}^{i-1}\sum_{\ell=i-2}^{n-3}(j-1)|\VPF_{n-3,(\ell)}| \ = \ \sum_{j=2}^{i-1}\sum_{\ell=j-1}^{n-3}(j-1)|\VPF_{n-3,(\ell)}|.
\end{equation*}
Expanding the summations and grouping common terms yields
\begin{equation*}
\sum_{\ell=1}^{i-3}\frac{\ell(\ell+1)}{2}|\VPF_{n-3, (\ell)}|+\frac{(i-2)(i-1)}{2}\sum_{\ell=i-2}^{n-3}|\VPF_{n-3, (\ell)}|.
\end{equation*}
When there is no such car $\ell$, i.e.~no cars prefer spot $n-2$, then $2\leq j\leq i-1$, $1\leq h\leq j-1$, and the number of parking functions is
\begin{equation*}
\sum_{j=2}^{i-1}\sum_{h=1}^{j-1}|\VPF_{n-3}|\ = \ \sum_{j=2}^{i-1}(j-1)|\VPF_{n-3}| \ = \ \frac{(i-1)(i-2)}{2}|\VPF_{n-3}|.
\end{equation*}
Hence, when three cars prefer spot $n-1$, the total number of parking functions is
\begin{equation*}
\frac{(i-1)(i-2)}{2}|\VPF_{n-3}|+\sum_{\ell=1}^{i-3}\frac{\ell(\ell+1)}{2}|\VPF_{n-3, (\ell)}|+\frac{(i-2)(i-1)}{2}\sum_{\ell=i-2}^{n-3}|\VPF_{n-3, (\ell)}|.
\end{equation*}
\item Two cars prefer spot $n-1$.\\
Let the first car preferring spot $n-1$ be the  car $j$ and the car taking spot $n-2$ be car $\ell$. 
We know by assumption that the second car preferring spot $n-1$ is car $i$. 
Since only two cars prefer spot $n-1$, then car $j$ parks in spot $n-1$,  car $\ell$ parks in spot $n-2$, which forces car $i$ to park in spot $n$. Hence, car $\ell$ precedes car $i$, i.e.~$\ell<i$. When car $\ell$ precedes car $j$, then $1\leq j\leq i-1$ and $1\leq \ell\leq j-1$. Since it is unclear whether  car $\ell$ prefers spot $n-2$ and parks in spot $n-2$, or prefers spot $n-3$ and parks in spot $n-2$, then the number of such parking functions is
\begin{equation*}
\sum_{j=1}^{i-1}\sum_{\ell=1}^{j-1}|\VPF_{n-2,\ell}|.
\end{equation*}
When car $\ell$ succeeds car $j$, then $1\leq j\leq i-1$ and $j+1\leq \ell\leq i-1$. Hence, the number of such parking functions is
\begin{equation*}
\sum_{j=1}^{i-1}\sum_{\ell=j}^{i-2}|\VPF_{n-2,\ell}|.
\end{equation*}
Notice that the summation goes from $\ell=j$ to $i-2$ instead of from $k=j+1$ to $i-1$ because car $j$ precedes car $\ell$. Then, when two cars prefer spot $n-1$, the total number of parking functions is
\begin{equation*}
\sum_{j=1}^{i-1}\sum_{\ell=1}^{j-1}|\VPF_{n-2,\ell}|+\sum_{j=1}^{i-1}\sum_{\ell=j}^{i-2}|\VPF_{n-2,\ell}| \ = \ \sum_{j=1}^{i-1}\sum_{\ell=1}^{i-2}|\VPF_{n-2,\ell}|\ = \ (i-1)\sum_{\ell=1}^{i-2}|\VPF_{n-2,\ell}|.
\end{equation*}
\item One car prefers spot $n-1$.\\
Since only one car prefers spot $n-1$, then by assumption this is car $i$. Since car $i$ parks in spot $n$, and no other cars prefer spot $n-1$, then spot $n-1$ is taken by a car preferring spot $n-2$. Suppose that is car $\ell$. Note that car $\ell$ precedes car $i$, and $1\leq \ell\leq i-1$. Hence, the number of ways to park  the cars in the remaining $n-1$ spots with car $\ell$ preferring spot $n-2$ and parking in spot $n-1$ is exactly $|\VPF_{n-1,[\ell]}|$. Hence, summing over all possible $\ell$'s, the number of parking functions is
\begin{equation*}
\sum_{\ell=1}^{i-1}|\VPF_{n-1,[\ell]}|.
\end{equation*}
\end{enumerate}
Adding the terms from each of the three cases above yields equation \eqref{V4} and completes the proof.
\end{proof}

\section{Monotonic vacillating Parking Functions}\label{mmp1}
In this section, we enumerate non-decreasing or non-increasing vacillating parking functions. Recall that given a parking function $\alpha=(a_1,a_2,\ldots,a_n)$ we say $\alpha$ is non-decreasing if $a_i\leq a_{i+1}$ for all $1\leq i\leq n-1$, and we say $\alpha$ is non-increasing if $a_{i+1}\leq a_{i}$ for all $1\leq i\leq n-1$. We separate our results based on which of these subsets we focus on. 
 
 \subsection{Non-decreasing Vacillating Parking Functions}
 Throughout, we let $\VPF_{n}^{\uparrow}$  denote the set of non-decreasing vacillating parking functions of length $n$.

\begin{lemma}\label{structure1}
If $\alpha = (a_1,a_2,\ldots,a_n) \in \VPF_{n}^{\uparrow}$, then $i-1 \leq a_i \leq i+1$ for all $i\in [n]$.
\end{lemma}

\begin{proof}
We proceed by induction. 
Let $i=1$ and suppose that $a_1 \geq 3$. Since $\alpha$ is non-decreasing, then no car can park at spot 1 and consequently $\alpha$ is not a parking function, resulting in a contradiction.
Assume for induction that $i-1\leq a_i\leq {i+1}$ for all $i\leq t$. We want to show that $t \leq a_{t+1} \leq t+2$. 

For the sake of contradiction, assume that $a_{t+1} \leq t-1$. From the inductive hypothesis, we know that $t-1 \leq a_t \leq t+1$. Since the parking function is non-decreasing, then $t-1\leq a_t\leq a_{t+1}$, and so $t-1\leq a_{t+1}$. However, by our assumption we also know that $a_t\leq t-1$. 
This then implies that, $a_t = a_{t+1} = t-1$. Since the first $t-1$ cars park between spot $1$ and spot $t$ inclusively, and $a_t = a_{t+1} = t-1$, then the first $t+1$ cars also park in between spot $1$ and spot $t$ inclusively. Since the number of cars exceed the number of spots, we arrive at a contradiction. Hence $a_{t+1} \geq t$.

Next, assume, for the sake of contradiction, that $a_{t+1} \geq t+3$. By the induction hypothesis, we know $i-1 \leq a_i \leq i + 1$ for all $1 \leq i \leq t$. 
Hence then the first $t$ cars park between spot $1$ and $t+2$ inclusively. As there are only $t$ cars parked among $t+2$ parking spots, let us further suppose that the two empty spots among the first $t+2$ spots are numbered $j$ and $k$, where $1 \leq j, k \leq t+2$. 
Since the parking function is non-decreasing, then $a_{t+2},a_{t+3}, \ldots,a_{n} \geq t+3$. In this case, spot $j$ and spot $k$ will not both be occupied by the cars numbered $t+2$ through $n$ (inclusive).
Therefore $\alpha$ is not a parking function, resulting in a contradiction.

This establishes that 
$t \leq a_{t+1} \leq t+2$ and completes our induction argument.
\end{proof}

The inequality in Lemma \ref{structure1} allows us to find a recursive formula for $|\VPF_{n}^{\uparrow}|$. 

\begin{theorem}\label{cardinality}If $n\geq 1$, then
\[ |\VPF_n^{\uparrow}| \ = \ 2\cdot|\VPF_{n-1}^{\uparrow}| + |\VPF_{n-2}^{\uparrow}|,\]
with $|\VPF_1^{\uparrow}|=1$ and $|\VPF_2^{\uparrow}|=3$.
\end{theorem}

\begin{proof}
Let $\alpha=(a_1,a_2,\ldots,a_n)\in\VPF_n^{\uparrow}$.
By Lemma \ref{structure1}, we know that there at most two preferences $a_i$ such that $a_i=n$. In fact $a_i=n$ only if $i=n-1$ or $i=n$. We consider the following three cases:
\begin{enumerate}[leftmargin=.3in]
    \item There is no $a_i$ such that $a_i = n$.
     Since $n-1 \leq a_n \leq n$ and $a_n \neq n$, then $a_n = n-1$. Since $a_n$ is predetermined and all other $a_i$'s are between $1$ and $n-1$ inclusively, then in this case the number of parking functions is equal to $|\VPF_{n-1}^{\uparrow}|$.
     \item There is one $a_i$ such that $a_i = n$.
     If there is only one $a_i = n$, then $i=n$ since $\alpha$ is non-decreasing. Since again $a_n$ is predetermined and all other $a_i$'s are between $1$ and $n-1$ inclusively, then in this case the number of parking functions is equal to $|\VPF_{n-1}^{\uparrow}|$.
     \item  There are two $a_i$ such that $a_i = n$.
     We know that the only two indices $i$ such that $a_i = n$ are $i=n-1$ and $i={n}$. In which case $a_{n-1}=a_{n}=n$. Then car $n-1$ and car $n$ will park in spot $n$ and spot $n-1$, respectively. 
     Hence, for $i\leq n-2$, we know $a_{i}\neq n-1$, as otherwise spot $n-1$ will be occupied before car $n$ arrives to park there. 
     Hence, $a_{i}\leq n-2$.
     Since $a_{n-1}$ and $a_n$ are predetermined and all other $a_i$'s are between $1$ and $n-2$ inclusively, then in this case the number of parking functions is equal to $|\VPF_{n-2}^{\uparrow}|$.
\end{enumerate}
The recursion follows from taking the sum over the count in the three above cases. For the initial conditions note that $|\VPF_{1}^\uparrow|=|\{(1)\}|=1$ and $|\VPF_2^\uparrow|=|\{(1,1),(1,2), (2,2)\}|=3$, as claimed.
\end{proof}

Before presenting a corollary to Theorem~\ref{cardinality}, we take a quick detour and provide some background on continued fractions.

\subsubsection{Background on continued fractions}\label{continued fractions}
A \textit{(simple) continued fraction} gives a way to express a real number $A$ as the sum of the integer part and the reciprocal of another real number $B$, from which the process iterates and $B$ is 
expressed as a sum of its integer part and another reciprocal. Iterating this process gives a possibly infinite  expression  of the form
\begin{align}\label{def:continued fraction}
A= a_1 +  \frac{1}{a_2 + \frac{1}{a_3 + \frac{1}{a_4 + \ldots}}},
\end{align}
where $a_1\in \mathbb{Z}$ and $a_i\in \mathbb{Z}_{\geq 0}$ for all $i\geq 2$.
It
is common to denote the continued fraction in \eqref{def:continued fraction} by $
    [a_1, a_2, a_3, a_4, 
\ldots]$.
Then $c_n=[a_1,\ldots,a_n]$ is called the 
\textit{$n$th convergent} of the continued fraction and $c_n=\frac{p_n}{q_n}$, where for $n\geq 3$, $p_n$ and $q_n$ satisfy
\begin{align}\label{recs:convergents}
    p_n&=a_np_{n-1}+p_{n-2}\qquad\mbox{ and }\qquad    q_n=a_nq_{n-1}+q_{n-2},
\end{align}
with initial values $p_1=a_1$, $q_1=1$ and $p_2=a_2a_1$, $q_2=a_2$.

It is now well-known that 
if $[a_1,a_2,\ldots]$ is finite, i.e., there exits $N\geq 1$ such that $a_i=0$ for all $i\geq N$, then the continued fraction represents a rational number.
On the other hand, if $[a_1,a_2,\ldots]$ is infinite, then it represents an irrational number \cite{rockett}.
Computing continued fractions of irrational numbers, such as $e$, $\pi$, and $\sqrt{2}$, has a long tradition in mathematics and mathematicians throughout history have contributed results to this study. 
For example, among Euler's numerous mathematical contributions, he established \cite{Euler} the continued fraction
$e=[2,1,2,1,1,4,1,1,6, 1,1,\ldots]$, which has a repeating pattern often denoted by $[2,1,\overline{2n,1,1}\,]_{n=1}^{\infty}$. A detailed proof of this result can be found in \cite{simplee}, with a shorter version provided in \cite{expandE}. In the case of~$\pi$,
Widž credits Huygens \cite{shortHistory} with establishing
$\pi=[3, 7, 15, 1, 292, 1, 1, 1, 2, 1,\ldots]$, for which no repeated pattern has been found. 
However, Lange gives an alternative continued fraction for $\pi$, whose  convergents do follow a nice pattern \cite[Theorem 1]{expandPi}. 

Using the recurrence relations in~\eqref{recs:convergents}, it is straightforward to show that the convergents $c_i=\frac{p_i}{q_i}$
of the continued fraction $\sqrt{2}=[1,2,2,2,\ldots]$ with sequnece given in \cite[\href{https://oeis.org/A040000}{A040000}]{OEIS},  satisfying $p_1=1$, $q_1=1$, $p_2=3$, $q_2=2$, and, for all $n\geq 3$,
 the recurrence relations 
\begin{align}
p_n&=2p_{n-1}+p_{n-2}\qquad\mbox{ and }\qquad
q_n=2q_{n-1}+q_{n-2}.\label{eq:convergents root 2}
\end{align}
The  sequence for $p_n$ in  \eqref{eq:convergents root 2} is given in \cite[\href{https://oeis.org/A001333}{A001333}]{OEIS} and next we show that these numerators have a combinatorial interpretation as the number of non-decreasing vacillating parking functions of length $n$.

\begin{corollary} The cardinality of
$|\VPF_{n}^{\uparrow}|$ is equal to the numerator of the $n$th convergent of the continued fraction of $\sqrt{2}$, which is also given by the recursive formula $p_n=2p_{n-1}+p_{n-2}$ with $p_1=1$ and $p_2=3$.
\end{corollary}

We conclude with a closed formula for the number of 
non-decreasing vacillating parking functions.
\begin{corollary}\label{sqrt2}
If $n\geq 1$, then
$|\VPF_{n}^{\uparrow}| \ = \ \frac{(1 + \sqrt{2})^n + (1 - \sqrt{2})^n }{2}$.
\end{corollary}
\begin{proof}
We prove this by induction. Note 
\begin{align*}
|\VPF_{1}^{\uparrow}|&=1=  \frac{(1 + \sqrt{2})^1 + (1 - \sqrt{2})^1 }{2},\\
\intertext{and}
|\VPF_{2}^{\uparrow}|&=3=\frac{(1 + \sqrt{2})^2 + (1 - \sqrt{2})^2 }{2}.
\end{align*}

Assume for induction that for all $3\leq i\leq n$, we have
\[|\VPF_{i}^{\uparrow}| \ = \ \frac{(1 + \sqrt{2})^i + (1 - \sqrt{2})^i }{2}.\]
Using Theorem~\ref{cardinality}, applying the induction to each term, and simplifying yields
\begin{align*}
|\VPF_{n}^{\uparrow}|&\ = \ 2\cdot|\VPF_{n-1}^{\uparrow}|+|\VPF_{n-2}^{\uparrow}|\\
&\ = \ 2\left(\frac{(1 + \sqrt{2})^{n-1} + (1 - \sqrt{2})^{n-1} }{2}\right)+\frac{(1 + \sqrt{2})^{n-2} + (1 - \sqrt{2})^{n-2} }{2}\\
&\ = \ \frac{(2(1+\sqrt{2})+1)(1+\sqrt{2})^{n-2}+(2(1-\sqrt{2})+1)(1-\sqrt{2})^{n-2}}{2}\\
&\ = \ \frac{(1+\sqrt{2})^2(1+\sqrt{2})^{n-2}+(1-\sqrt{2})^2(1-\sqrt{2})^{n-2}}{2}\\
&\ = \ \frac{(1 + \sqrt{2})^n + (1 - \sqrt{2})^n }{2},
\end{align*}
as claimed.
\end{proof}

\subsection{Non-increasing Vacillating Parking Functions}
 Throughout we let $\VPF_{n}^{\downarrow}$ to denote the set of non-increasing vacillating parking functions of length $n$.

We now present a result analogous to Theorem \ref{cardinality} consisting of a recursive formula for the number of non-increasing vacillating parking functions.
\begin{lemma}\label{structure2} If $\beta=({b_{1},b_{2},...,b_{n}})\in \VPF_{n}^{\downarrow}$, then $n-i \leq b_{i} \leq n+2-i$ for all $i\in [n]$.
\end{lemma}

\begin{proof}
We proceed by induction. Let $i=1$. Suppose that $b_{1}\leq n-2$. 
Since $\beta$ is non-increasing, then no car can park at spot $n$, contradicting that $\beta$ is a parking function. 
Hence, $b_{1}\geq n-1$. Moreover, $b_{1}\leq n+1$. Thus $n-1 \leq b_{1}\leq n+1$.
Assume for induction that $n-i\leq b_i\leq n+2-{i}$ for all $i\leq t$. 
We want to show that  $n-(t+1)\leq b_{t+1}\leq n+2-(t+1)$. 

For the sake of contradiction, assume that $b_{t+1}\geq n+2-t$. From the inductive hypothesis, we know that $n-t\leq b_{t}\leq n+2-t$. 
Since the parking function is non-increasing, then $b_{t+2}\leq b_{t+3}\leq \ldots \leq b_n \leq n-(t+2)$. This implies that the first $t+1$ cars park in between spot $n$ and spot $n+1-t$, inclusively. 
Since the number of cars exceed the number of spots, we arrive at a contradiction. 
Hence, $b_{t+1}\leq n+2-(t+1)$.

On the other hand, suppose, for the sake of contradiction, that $b_{t+1}\leq n-(t+2)$. Since $n-i \leq b_{i} \leq n+2-i$, for all $1\leq i\leq t$, then the first $t$ cars park between spot $n-(t+1)$ and $n$, inclusively. 
Let the two unoccupied spots between spots $n$ and spot $n-(t+1)$ be $l$ and $m$, where $n-1-t\leq l<m\leq n$. 
Since the parking function is non-increasing, then $b_{t+2}, b_{t+3}, \ldots,b_{n}\leq n-(t+2)$. 
In this case, spot $l$ and spot $m$ cannot be simultaneously occupied, therefore $\beta$ will not be a parking function, creating a contradiction.

Therefore, $n-i \leq b_{i} \leq n+2-i$, for all $i\in[n]$, as desired.
\end{proof}

\begin{theorem}\label{non-increasing}
If $n\geq 1$, then \[|\VPF_n^{\downarrow}| \ = \ 2\cdot|\VPF_{n-1}^{\downarrow}| + |\VPF_{n-3}^{\downarrow}|,\]
with $|\VPF_1^{\downarrow}|=1$, $|\VPF_2^{\downarrow}|=3$, and $|\VPF_3^{\downarrow}|=6$.
\end{theorem}

\begin{proof}
From Lemma \ref{structure2}, we know that there at most two $b_i$ such that $b_i=1$. 
In fact, $i=n-1$ or $i=n$. 
We consider the following three cases:
\begin{enumerate}[leftmargin=.3in]
    \item \label{case1}There is no $b_i$ such that $b_i = 1$.
     Since $1\leq b_{n} \leq 2$, then $b_{n}=2$. 
     Hence, $b_{n}$ is predetermined and all other $b_{i}$'s are between $2$ and $n$ inclusively, which is equivalent to parking $n-1$ cars between spot $1$ and spot $n-1$. Thus the number of parking functions is equal to $|\VPF_{n-1}^{\downarrow}|$.
     \item  There is one $b_i$ such that $b_i = 1$. 
    If there is only one $b_{i}=1$, then $i=n$ since $\beta$ is non-increasing. 
    We know that $b_{n}$ is predetermined and all other $b_{i}$ are between $2$ and $n$, but in this case the number of parking functions is not exactly equal to $|\VPF_{n-1}^{\downarrow}|$. 
    This is due to the existence of some parking functions $\gamma=(c_1,\ldots,c_{n-1})\in\VPF_{n-1}^{\downarrow}$ that disallow the $n$th car from parking in spot 1 after translation. For example, $(3, 1, 1)\in \VPF_{3}^{\downarrow}$ and after translation by 1 is $(4, 2, 2)$. However, $(4, 2, 2, 1)\notin \VPF_{4}^{\downarrow}$ since  car $3$ takes spot $1$ and car $4$ would be unable to park.     Moreover, $\gamma$ is such a parking function if and only if  
    $c_{n-1}=c_{n-2}=1$. 
    Thus we denote this set of parking functions as $\VPF_{n-1}^{\downarrow}(1,\{1,1\})$, and the number of parking functions in this case is $|\VPF_{n-1}^{\downarrow}(1,\{1,1\})|$.
     \item  There are two $b_{i}$ such that $b_{i} = 1$.
     We know that the only two $b_{i}$'s such that $b_{i}=1$ are $b_{n-1}$ and $b_{n}$. 
     Since $b_{n-1}=b_{n}=1$, then car $n-1$ and car $n$ park in spot $1$ and spot $2$, respectively. 
     Hence, $b_{n-2}\neq 2$, otherwise spot $2$ is taken before  car $n$ parks in it. 
     Since $2\leq b_{n-2}\leq 4$, then $b_{n-2}=3$ or $b_{n-2}=4$. 
     If $b_{n-2}=3$, then there are no other cars that prefer spot $3$, otherwise spot $2$ will be taken before car $n$ parks there. 
     Since $b_{n-1}$ and $b_{n-2}$ are fixed, then the number of parking functions in this case is equal to the number of parking functions $\gamma=(c_1,\ldots,c_{n-1})\in\VPF_{n-1}^{\downarrow}$ with 
     $c_{n-1}=c_{n-2}=1$, i.e. $|\VPF_{n-1}^{\downarrow}(1,\{1,1\})|$. Analogous to Case \eqref{case1}, if $b_{n-2}=4$, the number of parking functions is equal to the number of parking functions of length $n-3$, i.e. $|\VPF_{n-3}^{\downarrow}|$. 
\end{enumerate}
Taking the sum of the counts in each of the above cases yields
\begin{align*}
|\VPF_n^{\downarrow}| &\ = \ |\VPF_{n-1}^{\downarrow}|+\left(|\VPF_{n-1}^{\downarrow}|-|\VPF_{n-1}^{\downarrow}(1,\{1,1\})|\right)+\left(|\VPF_{n-3}^{\downarrow}|+|\VPF_{n-1}^{\downarrow}(1,\{1,1\})|\right)\\
&\ = \ 2\cdot|\VPF_{n-1}^{\downarrow}| + |\VPF_{n-3}^{\downarrow}|.
\end{align*}
To conclude, we note that 
$|\VPF_1^\downarrow|=|\{(1,1)\}|=1$,
$|\VPF_2^\downarrow|=|\{(1, 1),(2, 1),(2, 2)\}|=3$, and
$|\VPF_3^\downarrow|=|\{(2, 2, 2),(3, 1, 1),(3, 2, 1),(3, 2, 2),(3, 3, 1),(3, 3, 2)\}|=6$.
\end{proof}

We conclude with a closed formula for the number of non-increasing vacillating parking functions.

\begin{corollary}\label{delta}
If $n\geq 1$, then
\begin{equation*}
|\VPF_n^{\downarrow}|\ = \ \Delta_1\left(\frac{1}{r}\right)^{n+1}+\Delta_2\left(\frac{1}{\alpha+\beta i}\right)^{n+1}+\Delta_3\left(\frac{1}{\alpha-\beta i}\right)^{n+1},
\end{equation*}
where $r$, $\alpha+\beta i$, and $\alpha-\beta i$ denote the real and complex conjugate roots of the $f(x)=x^3+2x-1$ and
\begin{align*}
  \Delta_1 \ = \ \frac{r^3+r^2}{2r^3+1},\qquad   \Delta_2 \ = \ \frac{(\alpha^2-\beta^2+\alpha)+\beta(2\alpha+1)i}{-2\beta^2+2(\alpha-r)\beta i}, \qquad \Delta_3 \ = \ \frac{(\alpha^2-\beta^2+\alpha)-\beta(2\alpha+1)i}{-2\beta^2-2(\alpha-r)\beta i}. 
\end{align*}
\end{corollary}
\begin{proof}
Let $F(x)=\sum_{n\geq 0}|\VPF_{n}^{\downarrow}|x^n$. By Theorem \ref{non-increasing}, we have that
\begin{equation*}
\sum_{n\geq 0}|\VPF_{n+3}^{\downarrow}|x^n \ = \ 2\sum_{n\geq 0}|\VPF_{n+2}^{\downarrow}|x^n + \sum_{n\geq 0}|\VPF_{n}^{\downarrow}|x^n.
\end{equation*}
Hence
\begin{align}
\frac{F(x)-|\VPF_{2}^{\downarrow}|x^2-|\VPF_{1}^{\downarrow}|x}{x^3}\ = \ \frac{2F(x)-2|\VPF_{1}^{\downarrow}|x}{x^2}+F(x).\label{gen1}
\end{align}
Substituting the initial values $|\VPF_{1}^{\downarrow}|=1$, $|\VPF_{2}^{\downarrow}|=3$, and simplifying \eqref{gen1} we arrive at 
\begin{gather*}
F(x) \ = \ -\frac{x^2+x}{x^3+2x-1}.
\end{gather*}
Let
\begin{align}F(x)\ = \ \frac{\Delta_1}{r-x}+\frac{\Delta_2}{\alpha+\beta i-x}+\frac{\Delta_3}{\alpha-\beta i-x}.\label{neweq}
\end{align}
Then
$\Delta_1\ = \ F(x)(r-x)\big\rvert_{x=r}.$
By Vieta’s formula, $r(\alpha+\beta i)(\alpha-\beta i)=1$ implies $(\alpha+\beta i)(\alpha-\beta i)=1/r $ and $r+(\alpha+\beta i)(\alpha-\beta i)=0 $ implies $(\alpha+\beta i)(\alpha-\beta i)=-r$.
Hence, $\alpha+\beta i$ and $\alpha-\beta i$ are the zeros of $g(x)=x^2+rx+\frac{1}{r}$, which implies that $x^3+2x-1=-(r-x)(x^2+rx+\frac{1}{r})$. Thus,
\begin{align*}
\Delta_1 &\ = \ \frac{x^2+x}{x^2+rx+\frac{1}{r}}\bigg\rvert_{x=r}\ = \  \frac{r^3+r^2}{2r^3+1}.
\end{align*}
Similarly we establish that
\begin{align*}
\Delta_2 &\ = \ F(x)(\alpha+\beta i-x)\bigg\rvert_{x=\alpha+\beta i}
\ = \ -\frac{x^2+x}{(x-r)(\alpha-\beta i-x)}\bigg\rvert_{x=\alpha+\beta i} 
\ = \ \frac{(\alpha^2-\beta^2+\alpha)+\beta(2\alpha+1)i}{-2\beta^2+2(\alpha-r)\beta i}
\end{align*}
and
\begin{align*}
\Delta_3 &\ = \ F(x)(\alpha-\beta i-x)\bigg\rvert_{x=\alpha-\beta i} 
\ = \ -\frac{x^2+x}{(x-r)(\alpha+\beta i-x)}\bigg\rvert_{x=\alpha-\beta i} \ = \ \frac{(\alpha^2-\beta^2+\alpha)-\beta(2\alpha+1)i}{-2\beta^2-2(\alpha-r)\beta i}.
\end{align*}
Substituting $\Delta_1$, $\Delta_2$, and $\Delta_3$ into \eqref{neweq} and simplifying yields
\begin{align*}
F(x) &\ = \ \frac{\Delta_1/r}{1-x/r}+\frac{\Delta_2/(\alpha+\beta i)}{1-x/(\alpha+\beta i)}+\frac{\Delta_3/(\alpha-\beta i)}{1-x/(\alpha-\beta i)}\\
&\ = \ \frac{\Delta_1}{r}\sum_{n\geq 0}\left(\frac{x}{r}\right)^n+\frac{\Delta_2}{\alpha+\beta i}\sum_{n\geq 0}\left(\frac{x}{\alpha+\beta i}\right)^n+\frac{\Delta_3}{\alpha-\beta i}\sum_{n\geq n}\left(\frac{x}{\alpha-\beta i}\right)^n\\
&\ = \ \Delta_1\sum_{n\geq 0}\left(\frac{1}{r}\right)^{n+1}x^n+\Delta_2\sum_{n\geq 0}\left(\frac{1}{\alpha+\beta i}\right)^{n+1}x^n+\Delta_3\sum_{n\geq n}\left(\frac{1}{\alpha-\beta i}\right)^{n+1}x^n.
\end{align*}
Moreover, since $F(x)=\sum_{n\geq 0}|\VPF_{n}^{\downarrow}|x^n$, then
\[|\VPF_{n}^{\downarrow}| \ = \ \Delta_1\left(\frac{1}{r}\right)^{n+1}+\Delta_2\left(\frac{1}{\alpha+\beta i}\right)^{n+1}+\Delta_3\left(\frac{1}{\alpha-\beta i}\right)^{n+1}.\qedhere\]
\end{proof}

\section{Generalization to \texorpdfstring{$k$}{k}-vacillating Parking Functions}\label{Generalization}
In this section, we generalize our study to $k$-vacillating parking functions for $k\geq 1$. We will establish a connection between $k$-vacillating parking functions and vacillating parking functions.

\begin{theorem}\label{genearlgeneral}
Fix integers $1\leq k\leq n$. If $n=ka+b$, with $0\leq a$ and $0\leq b\leq k-1$, then \[|\VPF_n(k)| \ = \left(\
\frac{n!}{\prod_{t=0}^{k-1}\lfloor\frac{n+t}{k}\rfloor!}\right)\cdot|\VPF_{a+1}|^b\cdot|\VPF_{a}|^{k-b},\]
where $|\VPF_{0}|\coloneqq1$.
\end{theorem}

\begin{proof}
For any $1\leq j\leq k$, we define $I_j\coloneqq\{j, j+k, j+2k,\ldots, j+(m-1)k\}\subseteq [n]$ as the set of all spots congruent to $j$ modulo $k$. 
Let
 $I_j^{m}$ denote the set of $m$-tuples with entries in $I_j$. 
 Furthermore, let
$\VPF_{I_j}(k)\subseteq I_j^{m}$
be the set of parking functions in which all cars park in the spots indexed by the set $I_j$. 
Then, we can think of every $\gamma\in \VPF_{I_j}(k)$ as an $\alpha\in \VPF_{m}$, where each step is $1$ instead of $k$, and vice versa. Hence, $|\VPF_{I_j}(k)|= |\VPF_{m}|$. Moreover, if $1\leq j_{1},j_{2}\leq k$ and $j_{1}\neq j_{2}$, then $\VPF_{I_{j_1}}(k)$ and $\VPF_{I_{j_2}}(k)$ are independent. Since $m=a+1$ if $1\leq j\leq b$ and $m=a$ if $b+1\leq j\leq k$, 
this gives us $|\VPF_{a+1}|^b\cdot|\VPF_{a}|^{k-b}$ parking functions. 
However, since there is no restriction on where the entries of the smaller parking functions of length $a$ and $a+1$ can be placed in the parking function of length $n$, 
we need to multiply by a factor equal to the number of ways to place those entries within the larger tuple. 
This can be done in $\frac{n!}{(a!)^{k-b}(a+1)^b}$  ways. 
Note that this placement preserves the ordering of the entries in the smaller tuples and does not permute those entries.
Since $(a!)^{k-b}=\prod_{t=0}^{k-b-1}\lfloor\frac{n+t}{k}\rfloor!$ and $(a+1!)^{b}=\prod_{t=k-b}^{k-1}\lfloor\frac{n+t}{k}\rfloor!$, then
\[|\VPF_n(k)| \ = \ \left(\
\frac{n!}{\prod_{t=0}^{k-1}\lfloor\frac{n+t}{k}\rfloor!}\right)\cdot|\VPF_{a+1}|^b\cdot|\VPF_{a}|^{k-b}.\]
To conclude, note that the empty tuple, denoted by $()$, is the only element in $\VPF_0$. Thus $|\VPF_0|=1$.
\end{proof}

 

\section{Future Work}\label{sec:future}

There are many well-established methods to find  closed formulas for linear recurrences and systems of linear recurrences. However, it seems that none of the techniques can be applied to the system outlined in Theorem~\ref{thm:multirecurrence} because of the presence of the summations, the number of linear recurrences, as well as the number of variables involved. Hence, we pose the following.
\begin{problem}\label{problem1}
For $n\geq 1$, give a closed formula for the value of $|\VPF_{n}|$. 
If such a closed formula does not exist, then give an approximation for the value of $|\VPF_{n}|$.
\end{problem}

We also do not provide a closed formula for $|\VPF_{n}(k)|$, since, by Theorem~\ref{genearlgeneral}, this would rely on knowing a closed formula for $|\VPF_{n}|$. Hence, we pose the following.
\begin{problem}\label{problem2}
For $1\geq k\geq n$, give a closed formula for the value of $|\VPF_{n}( k)|$. 
If such a closed formula does not exist, then give an approximation for the value of $|\VPF_{n}(k)|$.
\end{problem}

The partition method in the proof of Theorem \ref{genearlgeneral} enumerates $k$-vacillating parking functions in terms of smaller parking functions. However, the same method fails if one wants to enumerate monotonic $k$-vacillating parking functions in terms of smaller monotonic parking functions, since monotonicity is not preserved after some translation and concatenation of monotonic parking functions. For example, if one wants to write a $2$-vacillating parking function of length $4$ in terms of $(2, 2)$ and $(2, 2)$ (both of which are vacillating parking functions), then after proper translation and concatenation the parking function becomes $(3, 4, 3, 4)$, which is not monotonic. Hence, we pose the following:

\begin{problem}
Fix integers $1\leq k\leq n$. Give a recurrence relation or a closed form formula for monotonic $k$-vacillating parking functions.
\end{problem}

It is well-known that any permutation of the entries of a parking function results in another parking functions. This property of parking functions is often referred to as ``permutation invariance.'' Vacillating parking functions are not permutation invariant. For example, $(1,1,2)\in\VPF_3$ as car 1 parks in spot 1, car 2 parks in spot 2, and car 3 parks in spot 3. However $(2,1,1)\notin\VPF_3$ since car 1 parks in spot 2, car 2 parks in spot 1, but car 3 fails to park. Hence we pose the following.

\begin{problem}
Fix integers $1\leq k\leq n$. Give a characterization and enumeration for the subset of $k$-vacillating parking functions that is permutation invariant.
\end{problem}

\section*{Acknowledgments}
P.~E.~Harris was supported through a Karen Uhlenbeck EDGE Fellowship.
The authors thank the William College's Science Center for funding that supported this research. The authors also thank J.~Carlos Mart\'{i}nez Mori for helpful conversations throughout the completion of this manuscript.

\bibliographystyle{plain}
\bibliography{Bibliography.bib}

\end{document}